\documentclass[11pt,reqno]{amsart}

\usepackage{a4wide}

\usepackage{macros}
\usepackage{hyperref}

\colorcommentstrue

\draftfalse

\usepackage{color}

\newcommand{\red}[1]{\textcolor[rgb]{1.00,0.00,0.00}{#1}}

\newcommand{\s}{{\small{$\Sigma$}}}

\begin{document}
\title{Shrinking Targets versus Recurrence: the quantitative theory }

\authorjason \authorbing \authordavid \authorsanju

\subjclass[2020]{Primary 37B20, Secondary 37D20}
\keywords{Shrinking target sets, recurrence, quantitative Borel-Cantelli}
\dedicatory{Dedicated to Bridget for making it to middle age! }

\begin{Abstract}
Let $X = [0,1]$, and let $T:X\to X$ be an expanding piecewise linear map sending each interval of linearity to $[0,1]$. For $\psi:\N\to\R_{\geq 0}$, $x\in X$, and $N\in\N$ we consider the recurrence counting function
\[
R(x,N;T,\psi) := \#\{1\leq n\leq N: d(T^n x, x) < \psi(n)\}.
\]
We show that for any $\epsilon > 0$ we have
\[
R(x,N;T,\psi) = \Psi(N)+O\left(\Psi^{1/2}(N) \ (\log\Psi(N))^{3/2+\varepsilon}\right)
\]
for $\mu$-almost all $x\in X$ and for all $N\in\N$, where $\Psi(N):= 2 \sum_{n=1}^N  \psi(n)$. We also prove a generalization of this result to higher dimensions.
\end{Abstract}
\maketitle

\section{Introduction: background and motivation}  \label{bgm}

Let $(X,d)$  be a compact metric space and $(X,\mathcal{A},\mu,T)$ be an ergodic probability measure preserving system.  Furthermore, given a real, positive function $\psi:\N\to\R_{\ge 0}$ let
$$
R(T,\psi) := \big\{ x \in X : d(T^n x, x) < \psi(n) \hbox{ for infinitely many }n\in \N  \big\}  $$
denote the associated \emph{recurrent set}, and given a point  $x_0 \in X$ let
$$
W( T, \psi) := \big\{ x \in X : d(T^nx, x_0) < \psi(n) \hbox{ for infinitely many }n\in \N  \big\}  $$
denote the  associated \emph{shrinking target set}.
If $\psi=c$ (a constant),   it follows from  from two foundational
results in dynamics (the Poincar\'{e} Recurrence  Theorem and  the Ergodic Theorem) that
$
  \mu( R(T,c) )   = 1 =
 \mu( W(T,c) )  \,     $.
In view of this, it is natural to ask:
what is the  $\mu$-measure of  the sets  if $\psi(n) \to 0 $ as $n \to \infty$?
In turn, whenever the $\mu$-measure is zero,  it is natural to ask about the Hausdorff dimension of the sets under consideration.   In this paper we will concentrate our attention solely to the $\mu$-measure question.

   Under certain mixing properties it known that if the `natural' measure sum diverges then both the recurrent and shrinking target sets are of full $\mu$-measure.  Indeed, if ``mixing'' is strong enough then is relatively straightforward  to obtain a quantitative strengthening of the full measure statement for  the shrinking target problem. This we briefly describe, mainly to motivate the problem for the  recurrent set but also for the sake of completeness. Note that for shrinking targets we interested in points $x \in X$ whose orbit $T^n(x)$  ``hits'' the balls $B_n:=B(x_0, \psi(n))$ an infinite number of times. It is worth emphasizing that the  centres $x_0$ of the balls  $B_n$ are fixed.  Now,  given $ N \in \N$ and $x \in X$,   consider the counting function
\begin{equation} \label{countdefst}
W(x,N; T, \psi) := \# \big\{ 1\le n \le   N :    T^nx \in B(x_0, \psi(n))  \} \, .
\end{equation}
Then under the assumption that $T$ is exponentially mixing with respect to $\mu$ (see \S\ref{setup} for the definition) we have the following statement.  It  tells us that if a `natural'  measure sum  diverges  then for  $\mu$--almost all $x \in X$  the orbit  `hits' the target sets $B_n$ the `expected' number of times.

\begin{theorem}\label{gencountthm}
Let $(X,\mathcal{A},\mu,T)$ be a measure-preserving dynamical system  and  suppose that $T$ is exponentially mixing with respect to $\mu$.  Let $\psi:\N\to\R_{\ge 0}$ be a  real, positive function.  Then, for any given $\varepsilon>0$, we have that
\begin{equation} \label{def_Psist}
W(x,N; T, \psi)=\Phi(N)+O\left(\Phi^{1/2}(N) \ (\log\Phi(N))^{3/2+\varepsilon}\right)
\end{equation}
\noindent for $\mu$-almost all $x\in X$, where
$$\Phi(N):=\sum_{n=1}^N \mu(B_n) \,  .
$$
\end{theorem}

We include a short proof in \S\ref{mechCR}.  Within the context of the above theorem, it is in fact  possible to replace the exponentially mixing assumption by the weaker notion of \s-mixing (short for summable mixing) -- see \cite[\S1.1]{LI2023108994}  for the more general statement and for connections to other works.   A simple consequence of Theorem~\ref{gencountthm} is that  $\lim_{N \to \infty} W(x,N;T,\psi) = \infty  $ for  $\mu$--almost all $x \in X$ if the measure sum $ \Phi := \lim_{N \to \infty} \Phi(N) $ diverges and is finite otherwise.  The upshot is that the theorem implies  the following zero-full measure criterion.

\begin{theorem} \label{STcountcor}
	Let $(X,\mathcal{A},\mu,T)$ be a measure-preserving dynamical system  and  suppose that $T$ is exponentially mixing with respect to $\mu$.   Let $\psi:\N\to\R_{\ge 0}$ be a  real, positive function.  Then
\begin{eqnarray}\label{appdiv}
		\mu\big(W(T,\psi)\big)=
		\begin{cases}
			0 &\text{if}\ \  \sum_{n=1}^\infty \mu\big(B_n\big)<\infty\\[2ex]
			1 &\text{if}\ \ \sum_{n=1}^\infty \mu\big(B_n\big)=\infty.
		\end{cases}
	\end{eqnarray}
\end{theorem}

\medskip

With various additional assumptions  on the  measure-preserving dynamical system, analogues of Theorem~\ref{STcountcor} for the recurrent set $R(T,\psi) $  have been established in numerous works  -- see \cite{Hussain2020DynamicalBL} and references within. It is worth noting that if one of the additional assumptions is that $\mu$ is $\delta$-Ahlfors regular (an assumption in \cite{Hussain2020DynamicalBL} and indeed the numerous works cited within)
then with reference to \eqref{appdiv} the measure of the ball $ \mu (B_n) $ is comparable to $\psi^{\delta}(n)$ and conclusion of the analogous zero-full measure criterion for $R(T,\psi) $ reads (as indeed it does in \cite{Hussain2020DynamicalBL}) as follows:
\begin{eqnarray}\label{HLSW}
		\mu\big(R (T,\psi)\big)=
		\begin{cases}
			0 &\text{if}\ \  \sum_{n=1}^\infty \psi^\delta(n)<\infty\\[2ex]
			1 &\text{if}\ \ \sum_{n=1}^\infty \psi^\delta(n)=\infty.
		\end{cases}
	\end{eqnarray}

\noindent  The papers \cite{Hussain2020DynamicalBL} and \cite{LI2023108994} are to some extent  the catalyst to this work. 
To the best of our knowledge, analogues of the stronger counting statement corresponding to  Theorem~\ref{gencountthm} seem to be rather thin in the literature
\footnote{Though the  recent   work of Persson \cite{Persson}, building upon earlier joint work  \cite{MR4675967},  aims to establish such a statement for maps of the unit interval it is somewhat different in nature.   Moreover,  with reference to \eqref{HLSW}, the assumptions made  in \cite{Persson} exclude the ``critical'' shrinking rate corresponding to $\psi(n) = n^{-\delta} $.  More precisely, it is required that  $\psi^\delta(n) >  (\log n)^{4+\epsilon}/n $ as well as other growth conditions. Subsequently, his work (with the various growth conditions) was extended by Sponheimer \cite{Sponheimer} to  more general dynamical systems including Axiom A diffeomorphisms. }  
This paper is a first  step attempting to address this imbalance.  With this is mind,   given $ N \in \N$ and $x \in X$,   consider the counting function
\begin{equation} \label{countdef}
R(x,N)=R(x,N; T, \psi) := \# \big\{ 1\le n \le   N :    d(T^nx, x) < \psi(n) \}.
\end{equation}
As alluded to in the definition, we will often simply write $R(x,N)$ for $R(x,N; T, \psi)$ since the other dependencies will be clear from the context and are usually fixed.  It is easily seen that the convergent  statement \eqref{HLSW} is equivalent to saying that if the sum converges, then  $\lim_{N \to \infty} R(x,N)  $ is finite  for  $\mu$--almost all $x \in X$ and there is nothing more to say in terms of the behaviour of counting function from a $\mu$--almost all point of view.  On the other hand, if the sum diverges, then  $\lim_{N \to \infty} R(x,N)$ is infinite    for  $\mu$--almost all $x \in X$ and it is both natural and desirable to attempt to  understand the asymptomatic behaviour of $  R(x,N) $  akin to  the conclusion of  Theorem~\ref{gencountthm} for the shrinking target  counting function. So what would  we expect the asymptomatic behaviour to look like for the recurrence counting function?    To answer this, it is beneficial to first understand the presence of the measure sum $\Phi(N)$ in Theorem~\ref{gencountthm}. In short, given $
W( T, \psi)$,  for each $n\in \N $ let
\begin{eqnarray}
  E_n   & :=  & \big\{ x \in X : d(T^nx, x_0) < \psi(n) \big\} \nonumber \\[1ex]
  & = & \big\{ x \in X :  T^nx \in B_n:=B(x_0, \psi(n)) \big\}  \ = \ T^{-n} (B_n) \label{souseful} \, .
\end{eqnarray}
By definition,
$
W( T, \psi) =  \limsup_{n \to \infty} E_n  \,
$
and observe that since $T$ is measure preserving \begin{equation} \label{invst} \mu (E_n) = \mu (B_n) \, .  \end{equation}
In other words, the measure sum appearing in Theorem~\ref{gencountthm} which determines the asymptomatic behaviour of the shrinking  target counting function  is nothing more than the measure of the fundamental sets $E_n$  associated with the $\limsup$ set under consideration.  With this in mind,  given $
R( T, \psi)$,  for each $n\in \N $ let
\begin{eqnarray}\label{def_An}
  A_n   & :=  & \big\{ x \in X : d(T^nx, x) < \psi(n) \big\} \nonumber \\[1ex]
  & = & \big\{ x \in X :  T^nx \in B(x, \psi(n)) \big\}   \, .
\end{eqnarray}
By definition,
$
R( T, \psi) =  \limsup_{n \to \infty} A_n  \,  \
$
and it would be reasonable to expect  (of course under suitable assumptions beyond just exponentially mixing) a quantitative  asymptotic statement of the following type: for $\mu$-almost all $x\in X$
$$
\lim_{N \to \infty}   \frac{R(x,N)}{\Phi(N)} \, = \, 1   \quad {\rm where \ }  \quad \Phi(N):=\sum_{n=1}^N \mu(A_n)
$$
and we assume that $\lim_{N \to \infty} \Phi(N)=\infty $.    Ideally, it would be  desirable to obtain the ``optimal''  statement in line with \eqref{def_Psist}; that is,  for any given $\varepsilon>0$
\begin{equation} \label{dreamrec}R(x,N) =\Phi(N)+O\left(\Phi^{1/2}(N) \ (\log\Phi(N))^{3/2+\varepsilon}\right)  \, ,
\end{equation}
for $\mu$-almost all $x\in X$.

\medskip

In this paper we establish \eqref{dreamrec} for a class of piecewise linear maps.  The main difficulty lies in the fact the sets $A_n$ associated with  $ R( T, \psi)$  cannot be expressed as the pre-image of ``nice'' sets let alone a ball as in the case of the sets $E_n$ associated with  $ W( T, \psi)$ -- see \eqref{souseful}.   This means that we are not able to directly exploit ``measure preserving'' and ``exponentially mixing'' as in the shrinking target problem.  This will be made absolutely evident in  \S\ref{mechCR} when we provide the proof of Theorem~\ref{gencountthm}.   At this stage, it is simply  worth noting  that due to \eqref{souseful} and the fact that $T$ is measure preserving,  the convergent part of Theorem~\ref{STcountcor} is a straightforward consequence of \eqref{invst}  and the (convergent) Borel-Cantelli Lemma.  For the recurrence problem,  establishing the analogous convergent part as in \eqref{HLSW} is not as direct even with extra assumptions.

\section{The setup and statement of results} \label{setup}

For the sake of clarity and familiarity,  we start with the setup in dimension one.  Let $T:[0,1]\to [0,1]$ be a piecewise linear map sending each interval of linearity to $[0,1]$. Let $J_m$ be an interval of linearity for $T^m$; that is, $J_m \in \PP_m $ where $\PP_m $ is the collection of cylinder sets of order $m$ that partition $[0,1]$.    
Let $K_{J_m} = (T^m)'|_{J_m}$. Then $T^m|_{J_m}$ is a similarity with dilatation factor $|K_{J_m}|$. In particular,  $(T^m)'(y) =K_{J_m} $ for any $y \in J_m$.
Moreover, the following are well known facts of this setup:

\begin{enumerate}
    \item[\textbf{[F1]}] $T$ is expanding. Thus there exists a constant $\lambda > 1 $ such that $|T'(x)|  \ge \lambda $ for all $ x \in [0,1]$. It follows that for any $m \in \N$
    \begin{equation}  \label{expanding constant}
   | K_{J_m} |  \ge \lambda^m \, .  \end{equation}
    \item[\textbf{[F2]}] $T$ is a measure preserving transformation with respect to one-dimensional Lebsegue measure $\mu$; i.e. $\mu$ is a $T$-invariant probability measure. For any measurable set $F \subseteq J_m $, due to linearity, we have that
 $$
 \mu(T^m(F) )    =  K_{J_m}  \mu(S).
 $$
    \\
    \item[\textbf{[F3]}] $T$ is a exponentially mixing  with respect to $\mu$; i.e. there exists a constant $0<\gamma<1$ such that for any $ n \in \N$, and  any ball $ B \in [0,1]$ and measurable set $F\in [0,1]$ \,
\begin{equation}\label{mixing1d}
\mu(B\cap T^{-n}(F))=\mu(B)\mu(F)+O(\gamma^n)\mu(F)  \, .
\end{equation}
\end{enumerate}

\medskip

\begin{remark} 
Although well known, the  above exponential mixing fact  [F3] will also follow as a consequence of a more general mixing statement  established in this paper (namely, Proposition~\ref{propositionexpmix} appearing in \S\ref{prelimhd}).
\end{remark} 

The following constitutes our main one dimensional result.

\begin{theorem} \label{1dimthm}
    Let $T:[0,1]\to [0,1]$ be a piecewise linear map sending each interval of linearity to $[0,1]$. Let $\psi:\N\to\R_{\ge 0}$ be a  real, positive function. Then, for any given $\varepsilon>0$, we have that
\begin{equation} \label{def_Psi}
R(x,N)=\Psi(N)+O\left(\Psi^{1/2}(N) \ (\log\Psi(N))^{3/2+\varepsilon}\right)
\end{equation}
\noindent for $\mu$-almost all $x\in X$, where
$$\Psi(N):= 2 \sum_{n=1}^N  \psi(n) \,  .
$$
\end{theorem}

\vspace*{3ex}

 In higher dimensions, we consider the following generalization of the above set up.

 \begin{itemize}
   \item Let $T:[0,1]^d\to [0,1]^d$ be a piecewise linear map such that each domain of linearity $J_1$ is a coordinate-parallel rectangle that $T$ sends to $X=[0,1]^d$ and such that $T'|_{J_1}$ is a diagonal matrix.   For general  $m \in \N$, the rectangle of linearity  $J_m  \in \PP_m $  where $ \PP_m $ is the collection of cylinder sets $J_m$ of order $m$ that partition $X=[0,1]^d$. Throughout,   we use the standard notation $D_xT$ to mean the derivative of $T$ at a point $x \in X$ but since
$D_xT$ is the same for any $x\in J_m$, we simply write $T'|_{J_m}$. \\

\item Let $ \mu$ be $d$-dimensional Lebesgue measure. \\ 

 \item  For each $i=1,\ldots,d$ let $\psi_i:\N\to\R_{\ge 0}$ be a  real, positive function and let \begin{equation} \label{profpsi}\psi(n) := \prod_{i=1}^d \psi_i(n)  \, . \end{equation}
 In turn, for   $ N \in \N$ and $x:=\{x_1, \ldots, x_d \}  \in X$,   consider the counting function
\begin{equation} \label{countdefhd}
R(x,N) := \# \Big\{ 1\le n \le   N :     \dist(x_i,T^n(x_i)) \leq \psi_i(n) \all \  1 \le i \le d \Big\} .
\end{equation}
 \end{itemize}

\noindent Examples of well know dynamical systems satisfying our setup include the tent map, the L\"{u}roth map and expanding, diagonal integer matrix transformations of the $d$-dimesnional torus $\T^d$. The following constitutes our main result.

\begin{theorem} \label{ddimthmhd}
    Let $T:[0,1]^d\to [0,1]^d$ and $\psi_i:\N\to\R_{\ge 0}$ be as above. Then, for any given $\varepsilon>0$, we have that
\begin{equation} \label{def_Psihd}
R(x,N)=\Psi(N)+O\left(\Psi^{1/2}(N) \ (\log\Psi(N))^{3/2+\varepsilon}\right)
\end{equation}
\noindent for $\mu$-almost all $x\in X$, where
$$\Psi(N):= 2^d \sum_{n=1}^N  \psi(n) \,  .
$$
\end{theorem}

\noindent For the sake of completeness,  we point out  that the analogues of the one-dimensional facts [F1] - [F3] are valid within the higher dimensional framework of the theorem and are essential for its proof (see~\S\ref{secproof}).

\medskip 

\begin{remark} \label{worthit}
  It seems that even for the simplest  one dimensional map  $ T: x \to 2x $ modulo one, the theorem appears to be new.  As mentioned in the footnote in \S\ref{bgm}, the related works of Persson \cite{Persson} and  Sponheimer \cite{Sponheimer}, although applicable to a larger class of dynamical systems, impose various  growth conditions on $\psi$. Also, their approach does not provide an error term for the asymptotic behavior of $R(x,N)$ and in the  higher dimensional work of Sponeheimer \cite{Sponheimer}  it is assumed that $\psi_1 =  \ldots = \psi_d$.      When adding  the finishing  touches  to  this paper, we became aware of the work of He \cite{He}.  By imposing  stronger conditions on the underlying dynamical system than those in \cite{Persson,Sponheimer},  He  removes the various growth conditions  and   proves  an asymptotic formula   but  again without an error term.  His  conditions on the dynamical system are less restrictive than ours (so, for example,  it includes the Gauss map) and in higher dimensions, like us, He does not require that $\psi_1 =  \ldots = \psi_d$.  We emphasise that compared to above three works, although the  
   dynamical systems we consider are more restrictive,   we obtain a  stronger asymptotic  statement that is in line with the analogous shrinking target result  and the proof is both transparent and short; again more in line with the shrinking target situation.  
\end{remark}

\section{A mechanism for establishing counting results}  \label{mechCR}

The following  statement~\cite[Lemma~1.5]{Harman} represents an important  tool in the theory of metric Diophantine approximation for establishing counting statements.  It has its bases in the familiar variance method of probability theory and can be viewed as the quantitative form of the (divergence)  Borel-Cantelli Lemma \cite[Lemma~2.2]{BRV}.

\begin{lemma} \label{ebc}
Let $(X,\mathcal{A},\mu)$ be a probability space, let $(f_n)_{n \in \N}$ be a sequence of non-negative $\mu$-measurable functions defined on $X$, and $(c_n)_{n \in \N },\ (\phi_n)_{n  \in \N}$ be sequences of real numbers  such that
$$ 0\leq c_n \leq \phi_n \hspace{7mm} (n=1,2,\ldots).  $$

\noindent 
Suppose that for arbitrary  $a,b \in \N$ with $a <  b$, we have
\begin{equation} \label{ebc_condition1}
\int_{X} \left(\sum_{n=a}^{b} \big( f_n(x) -  c_n \big) \right)^2\mathrm{d}\mu(x)\, \leq\,  C\!\sum_{n=a}^{b}\phi_n
\end{equation}

\noindent for an absolute constant $C>0$. Then, for any given $\varepsilon>0$,  we have
\begin{equation} \label{ebc_conclusion}
\sum_{n=1}^N f_n(x)\, =\, \sum_{n=1}^{N}c_n\, +\, O\left(\Phi(N)^{1/2}\log^{\frac{3}{2}+\varepsilon}\Phi(N)+\max_{1\leq k\leq N}c_k\right)
\end{equation}
\noindent for $\mu$-almost all $x\in X$, where $
\Phi(N):= \sum\limits_{n=1}^{N}\phi_n
$. 
\end{lemma}

\begin{remark}
Note that in statistical terms, if the sequence $c_n$ is the mean of $f_n$; i.e.
$$
c_n = \int_{X} f_n(x) \mathrm{d}\mu(x) \, ,
$$
then the l.h.s.~of~\eqref{ebc_condition1} is simply the variance ${\rm Var}(Z_{a,b}) $ of the random variable  $$Z_{a,b}=Z_{a,b}(x):=\sum_{n=a}^{b}f_n(x) \, . $$ In particular,
$
 {\rm Var}(Z_{a,b})  = \E(Z^2_{a,b})  -  \E(Z_{a,b})^2   \ \
{\rm where \ }   \;
 \E(Z_{a,b}) = \int_X Z_{a,b}(x) \mathrm{d}\mu(x)  \, .
$

\end{remark}

\medskip
\subsection{The proof of Theorem~\ref{ddimthmhd} modulo  `holes'}  \label{twoholes}

With the goal of establishing Theorem~\ref{ddimthmhd} in mind, given a real, positive function $\psi:\N\to\R_{\ge 0}$  we consider Lemma \ref{ebc}  with
\begin{equation} \label{Harman_choice_parameters}
X := [0,1]^d \, , \qquad f_n(x):= \chi_{A_{n}}\!(x)  \qquad {\rm and } \qquad
 c_n:= \phi_n:= \mu(A_n) \, ,
\end{equation}
where   $  \chi_{A_{n}}\!\!$   is the characteristic function  of the set  $A_n$ ($n \in \N$) given by
\begin{eqnarray}\label{def_Anhd}
  A_n & :=   &   \Big\{ x=\{x_1, \ldots, x_d \} \in [0,1]^d:     \dist(x_i,T^n(x)_i) \leq \psi_i(n) \all \  1 \le i \le d \Big\} \nonumber \\[2ex]
  & = &   \Big\{ x \in [0,1]^d:  T^n(x) \in  \textstyle{\prod_{i=1}^d} B( x_i , \psi_i(n) ) \Big\}   \, . 
\end{eqnarray}
\\
Then, clearly for any $ x \in X$ and $N \in \N$  we have that the
$$
{\rm l.h.s. \ of \ } \eqref{ebc_conclusion}  \ =   \ R(x, N) \, ,
$$
where $R(x,N)$ is the counting function given by \eqref{countdefhd}.  Also,  observe that the main term on the ${\rm r.h.s. \ of }$ \eqref{ebc_conclusion} is
\begin{equation}  \label{needed}
\Phi(N) :=  \sum\limits_{n=1}^{N}\mu (A_n) \, . \end{equation}     Furthermore, it is easily verified  that for any $a,b \in \N$ with $a <  b$

\begin{equation*}
\begin{aligned}
\left(\sum_{n=a}^{b}(f_n(x)-c_n) \right)^2  &=  \ \left(\sum_{n=a}^{b}f_n(x)\right)^2 \, + \, \left(\sum_{n=a}^{b}c_n\right)^2 -  \ 2\sum_{n=a}^{b}f_n(x) \cdot \sum_{n=a}^{b}c_n \\[2ex]
   =  \  \sum_{n=a}^{b}f_n(x)  \  &+  \ 2\mathop{\sum\sum}_{a\leq m < n\leq b}f_m(x)f_n(x) \  +  \  \left(\sum_{n=a}^{b}c_n\right)^2  \ -  \ 2\sum_{n=a}^{b}c_n\cdot\sum_{n=a}^{b}f_n(x) \,  ,  \vspace{2mm}
\end{aligned}
\end{equation*}

\vspace{2mm}

\noindent and so it follows that
\vspace{2mm}
\begin{equation} \label{integral_square_estimate}
\begin{aligned}
  {\rm l.h.s. \ of \ } \eqref{ebc_condition1} 
  \ =  \   \sum_{n=a}^{b}\mu(A_n)   \, + \, 2 &\mathop{\sum\sum}_{a\leq m<n\leq b}  \mu(A_m\cap A_n)  \  -  \  \left(\sum_{n=a}^{b}\mu(A_n) \right)^2
\end{aligned}
\end{equation}


\vspace{2mm}

\noindent where $\mu$ is $d$-dimensional Lebesgue measure  supported on $X$.
The upshot of this is that in view of Lemma \ref{ebc},  the proof of  Theorem~\ref{ddimthmhd}  boils down to  `appropriately' estimating  the right hand side  of \eqref{integral_square_estimate} and showing that  $\Phi(N)$  can be replaced by $\Psi(N)$.   Regarding the former, estimating the measure of the intersection of the sets  $A_n $ is where the main difficulty lies. In short we need to show that the  sets $A_n $ are pairwise independent on average with an acceptable error term. The following is  at the heart of proving the desired counting result.

\begin{proposition} \label{measinterSV} For arbitrary $a, b \in \N $ with $a < b$,
  \begin{equation*}
       2\mathop{\sum\sum}_{a\leq m<n\leq b}  \mu(A_m \cap A_n) \le  \Big( \sum_{n=a}^{b}  \mu(A_n) \Big)^2  \ + \ O\Big(  \sum_{n=a}^{b}  \mu(A_n)    \ \Big) + O(1)   \, .
  \end{equation*}
\end{proposition}

With this statement at hand, it follows that
\begin{equation*} \label{rhssv}
\begin{aligned}
  {\rm r.h.s. \ of \ } \eqref{integral_square_estimate} 
  \ \ll   \   \sum_{n=a}^{b}\mu(A_n) =:  \sum_{n=a}^{b} c_n   \, .
\end{aligned}
\end{equation*}
This shows that \eqref{ebc_condition1} is satisfied with $\phi_n := c_n$ and on applying Lemma \ref{ebc} we obtain the statement: for any given $\varepsilon>0$
\begin{equation}  \label{giveme}
R(x,N)=\Phi(N)+O\left(\Phi^{1/2}(N) \ (\log\Phi(N))^{3/2+\varepsilon}\right)  \, ,
\end{equation}
\noindent for $\mu$-almost all $x\in X$.  Observe that this is precisely the desired statement \eqref{def_Psihd} apart from the fact that  $\Psi(N) $ is replaced by $\Phi(N)$.   It is easily versified that the  following estimate regarding the  measure of $A_n$ on average enables us to replaced $\Phi(N)$ by $\Psi(N) $ in \eqref{giveme} and thus complete the proof of Theorem~\ref{ddimthmhd}.

\begin{proposition} \label{measAnSV} For arbitrary $N \in \N$,
  \begin{equation*}
     \Phi(N):=  \sum_{n=1}^{N} \mu(A_n) = 2^d \sum_{n=1}^{N} \psi(n) =: \Psi(N)   \ + \ O(1)   \, .
  \end{equation*}
\end{proposition}

The upshot of the above discussion is that the proof of Theorem~\ref{ddimthmhd} is reduced to establishing Propositions~\ref{measinterSV} $\&$ \ref{measAnSV}.  This will be the subject of \S\ref{secproof}.   First we make a slight detour and show that in the shrinking target case,  due to the fact  the sets $E_n$ associated with  $ W( T, \psi)$  can be expressed as the pre-image of a ball (see \eqref{souseful}), we are directly able  to exploit ``measure preserving'' and ``exponentially mixing'' to establishing the analogue of Proposition~\ref{measinterSV}.

\subsection{The proof of Theorem~\ref{gencountthm}}

For each $n \in \N$, let  $E_n$  be the set defined in \eqref{souseful}.
It  is not difficult to see that on replacing $A_n$ by $E_n$ in \S\ref{twoholes}, the proof of Theorem~\ref{gencountthm} is reduced to showing  the  following pairwise independent on average statement.

\begin{proposition} \label{measinterSVST} For arbitrary $a, b \in \N $ with $a < b$,
  \begin{equation*}
       2\mathop{\sum\sum}_{a\leq m<n\leq b}  \mu(E_m \cap E_n) \le  \Big( \sum_{n=a}^{b}  \mu(E_n) \Big)^2  \ + \ O\Big(  \sum_{n=a}^{b}  \mu(E_n)    \ \Big)   \, .
  \end{equation*}
\end{proposition}

\begin{proof}
Let  $m < n  $.  Then in view of  \eqref{souseful} and the fact that $T$ is measure preserving and exponentially mixing  with respect to $\mu$, it follows that there exist constants $C > 0 $ and $\gamma \in (0, 1) $ such that
\begin{eqnarray*} \label{rhssveasy}
\mu(E_m \cap E_n)  &  =  &  \mu\Big(T^{-m}(B_m)  \cap T^{-n}(B_n)\Big)  =  \mu\Big(B_m  \cap T^{-(n-m)}(B_n)\Big) \\[1ex]
& \le  &  \mu\big(B_m \big)  \mu\big(B_n \big)   +   C  \gamma^{n-m} \mu\big(B_n \big)  \ = \  \mu\big(E_m \big)  \mu\big(E_n \big)   +   C  \gamma^{n-m} \mu\big(E_n \big)  \, .
\end{eqnarray*}
In turn, it follows that
\begin{eqnarray*}
2\mathop{\sum\sum}_{a\leq m<n\leq b}  \mu(E_m \cap E_n)   &  \le  &
\Big( \sum_{n=a}^{b}  \mu(E_n) \Big)^2  \ + \ C \sum_{n=a}^{b}  \mu(E_n) \sum_{m=1}^{\infty}  \gamma^m   \  .
\end{eqnarray*}
This completes the proof since the sum involving $\gamma$ is convergent.
\end{proof}

\section{Completing the proof of Theorem~\ref{ddimthmhd}:  filling in the holes.   }  \label{secproof}

Throughout,  $X:=[0,1]^d$ and for each  $n \in N$ the set $A_n$ is given by \eqref{def_Anhd}.  In view of \S\ref{twoholes}, in order to complete the proof of Theorem~\ref{ddimthmhd} we need to establish  Propositions~\ref{measinterSV} $\&$ \ref{measAnSV}.  We start with a few preliminaries.

\subsection{Preliminaries} \label{prelimhd}
Recall,  for   $m \in \N$,   $J_m  \in \PP_m $   is a rectangle of linearity for $T^m$ 
and we are assuming that $(T^m)'|_{J_m}$ is a diagonal matrix.   With this in mind, 
for each $i=1,\ldots,d$,  let $K_{J_m}^{(i)}$ be the $(i,i)$-th component of the diagonal matrix $(T^m)'|_{J_m}$  and let $K_{J_m} = \prod_i K_{J_m}^{(i)}$.  Then,  given the setup,  we have that  $ |K_{J_m}^{(i)} |  = 1/ \mu_1( J_m^{(i)} )   $ for each  $i=1,\ldots,d$, where $\mu_1$ is one-dimensional Lebesgue measure. Hence,    $ |K_{J_m} |  = 1/ \mu( {J_m} )   $ for each  $J_m  \in \PP_m $ where  $ \mu$ is  $d$-dimensional Lebesgue measure.     Moreover, it follows that there exists $\lambda > 1 $ such that for $m \in \N$ and  $J_m  \in \PP_m $
\begin{equation}
\label{gammadef**}
|K_{J_m}^{(i)}|\ge   \lambda^m        \qquad     (1 \le i \le d)  \, .  
\end{equation}
Thus, it follows that   $T$ is expanding   and indeed that  $T$ is  measure preserving with respect to $\mu$.  In terms of the analogues of the one-dimensional facts [F1] - [F3] stated at  the start of \S\ref{setup}, it remains to establish [F3].

\begin{proposition}
\label{propositionexpmix}
$T$ is rectangular exponentially mixing with respect to $\mu$; i.e. there exists  a constant  $0<\gamma<1$ such that for any $ n \in \N$, and  any rectangle  $ E \in [0,1]$ and measurable set $F\in [0,1]$,
\begin{equation}\label{mixing}
\mu(E\cap T^{-n}(F))=\mu(E)\mu(F)+O(\gamma^n)\mu(F)  \, .
\end{equation}
\end{proposition}

\begin{proof}
Let $E$ be a rectangle and $F\subset [0,1]^d$ be any set. First suppose that $E = J_n \in \PP_n$. Then by the change of variables formula
\[
\mu(E\cap T^{-n}(F)) = \mu((T^n|_E)^{-1}(F))
= |\det((T^n|_E)^{-1})'| \mu(F) = \mu(E) \mu(F).
\]
Note that  $\det((T^n|_E)^{-1})'    = 1/ K_{J_m} $.
Given a  set $F$, we denote by $\del F$ its boundary  and by $ \NN(F,\delta)$  its $\delta$-neighbourhood. Now let $E$ be any rectangle.  It  follows that
\begin{align*}
\mu(E\cap T^{-n}F)
&\le \sum_{\substack{J_n \\ J_n\cap E \neq \emptyset}} \mu(J_n\cap T^{-n}(F))\\[1ex]
&= \sum_{\substack{J_n \\ J_n\cap E \neq \emptyset}} \mu(J_n)\mu(F)\\[1ex]
&\le \mu(E)\mu(F) + \sum_{\substack{J_n \\ J_n\cap \del E \neq \emptyset}} \mu(J_n)\mu(F)\\[1ex]
&\le \mu(E)\mu(F) + \mu\big(\NN(\del E, \gamma^n)\big)\mu(F)\\[2ex]
&\le \mu(E)\mu(F) + 4d  \gamma^n\mu(F)  \, ,
\end{align*}
where  $0<\gamma := 1/ \lambda  < 1$  and $\lambda$ satisfies  \eqref{gammadef**}.
\end{proof}

\noindent It is evident that we can take $\gamma := 1/ \lambda  $ in the statement of Proposition~\ref{propositionexpmix} and so we can rewrite \eqref{gammadef**} as: 
\begin{equation}
\label{gammadefsv}
|K_{J_m}^{(i)}|\ge   \gamma^{-m}        \qquad     (1 \le i \le d)  \, .  
\end{equation}

\medskip

The following  is an  extremely useful statement   and is a straightforward application of the triangle inequality.    In short, it provides a   mechanism for ``locally'' representing $A_n$ as the inverse image of a coordinate-parallel rectangle.  In turn this allows us to exploit mixing.

\begin{lemma}  \label{usefullemmahigherdim}
Let $I := \prod_{i=1}^d B(z_i, r_i)$ be a coordinate-parallel rectangle centreed at $z \in X$.  Then for any $m \in \N$  with  $\psi_i(m) > r_i$ and any set $S \subseteq B$
\begin{equation*}
S \cap  T^{-m} \left( \prod_{i=1}^d B \left(z_i,  \psi_i(m) - r_i \right) \right)   \ \subset \   S \cap A_m     \ \subset \   S \cap  T^{-m} \left( \prod_{i=1}^d B (z_i,  \psi_i(m) + r_i ) \right)
\end{equation*}
\end{lemma}

\begin{proof}
Let $ x \in S \cap A_m$.  Then, by definition  $\dist (x_i,z_i) \le r_i $ and $\dist (T^m(x)_i,x_i) <  \psi_i(m)$.   So by the triangle inequality, it follows (irrespective of whether or not  $\psi_i(m) > r_i$) that
$$  \dist\big(T^m(x)_i, z_i \big)       \le         \dist\big(T^m(x)_i, x_i  \big)  +  \dist\big(x_i, z_i  \big)  \ \le  \  \psi_i(m) + r_i \, .
$$
In other words, $T^m(x) \in \prod_{i=1}^d B \big( z_i, \psi_i(m) + r_i  \big)$ and so $ x \in \ T^{-m} \big(\prod_{i=1}^d B \big( z_i, \psi_i(m) + r_i  \big) \big) $. Hence,
$$S \cap A_m     \ \subset \   S \cap  T^{-m} \big( B \left(z,  \psi(m) + r  \right) \big)\,  $$
which is precisely the right hand side of the desired statement.
For the left hand side, let $x \in  S \cap  T^{-m} \big( \prod_{i=1}^d B (z_i,  \psi_i(m) - r_i ) \big)$.   Then,  by definition  $\dist (x_i,z_i) \le r_i $ and $\dist (T^m(x)_i,z_i)  ) <  \psi_i(m) -r_i $.  So by the triangle inequality, it follows that
$$  \dist\big(T^m(x)_i, x_i \big)       \le         \dist\big(T^m(x)_i, z_i  \big)  +  \dist\big(x_i, z_i  \big)  \ \le  \  \psi_i(m) \, .
$$
Thus,  $ x \in S \cap A_m$ and this  establishes the left hand side of the desired statement.
\end{proof}

We now consider a special case of the previous lemma in which  $B=S=J_m$

\begin{lemma}
\label{lemmaIJm}
For  any $m \in \N$ and each rectangle of linearity $J_m$ for $T^m$,  we have that $
I_{J_m} := J_m\cap A_m$  is a rectangle  whose  $i$-dimensional sidelength $\ell_n^{(i)}$  satsifies  
\[
\ell_n^{(i)}  \ \leq \  \frac{2\psi_i(m)}{|K_{J_m}^{(i)} - 1|}\cdot
\]
\end{lemma}
\begin{proof}
Since $T^m$ is linear on $J_m$   and $(T^m)'|_{J_m}$ is a diagonal matrix, it follows that for each $J_m$ there exists some $z\in \R^d$ such that 
\[
T^m(x) = (K_{J_m} x_1 - z_1 , \ldots ,  K_{J_m} x_d - z_d )    \quad \forall  \  x\in J_m   \, . 
\]
 Thus
\begin{align*}
x\in J_m\cap A_m
&\Longleftrightarrow  \ \forall    \ i=1,\ldots,d \; \\  & \hspace*{13ex} x_i\in J_m^{(i)}  \ \text{ and }  \  \dist(x_i,K_{J_m}^{(i)} x_i - z_i) < \psi_i(m)\\[2ex] 
&\Longleftrightarrow   \ \forall \  i=1,\ldots,d \; \\  & \hspace*{13ex} x_i\in J_m^{(i)}  \  \text{ and }  \  \dist(z_i,(K_{J_m}^{(i)} - 1) x_i) < \psi_i(m)\\[2ex]
&\Longleftrightarrow   \ \forall \  i=1,\ldots,d \; \\  & \hspace*{13ex} x_i\in I_{J_m}^{(i)} := J_m^{(i)} \cap B \Big((K_{J_m}^{(i)} - 1)^{-1} z_i, |K_{J_m}^{(i)} - 1|^{-1}\psi_i(m) \Big)  \, . 
\end{align*}
The upshot is that $I_{J_m} := J_m\cap A_m$ is a rectangle $\prod_{i=1}^d I_{J_m}^{(i)}$  and the sidelength $\ell_n^{(i)}$ of   the  $i$-dimensional side  $I_{J_m}^{(i)}$  is bounded above by  $ 2 |K_{J_m}^{(i)} - 1|^{-1}\psi_i(m)  $  as claimed. . 
\end{proof}

\subsection{Calculating the measure of $A_n$ on average}  \label{hd}
The goal of this section is to prove the   Proposition~\ref{measAnSV}.

Let $\{\epsilon_n\}_{n \in \N}$ be a sequence of real numbers in $(0,1)$  that we specify later.  At this point it is enough to observe that $\epsilon_n \psi_i(n) < \psi_i(n) < 1 $  for any $ n \in \N$ and $1 \le i \le d$.  With this in mind, for each $1 \le i \le d$, we split  $[0,1]$ into 
 $$\left\lceil \frac{1}{\epsilon_n\psi_i(n)}\right\rceil$$ disjoint intervals of equal length 
\begin{equation}  \label{slv1}
{l}_n^{(i)} := \frac{1}{\left\lceil
\frac{1}{\epsilon_n \psi_i(n)}
\right\rceil}    \, . 
\end{equation}
Taking the product, we can write $X=[0,1]^d$ as the union of
\[
\prod_{i=1}^d \left\lceil \frac 1{\epsilon_n\psi_i(n)}\right\rceil
\]
disjoint rectangles whose $i$-dimensional sidelength ${l}_n^{(i)}$ is given by \eqref{slv1} and by definition satisfies 
\begin{equation}  \label{slv22}
{l}_n^{(i)}   \leq  \epsilon_n \psi_i(n) \, . 
\end{equation}
Let   $\II$ denote the corresponding collection of rectangles.  Fix a rectangle $I = \prod_i I^{(i)} \in \II$ and let $z_I =(z_{I,i},   \ldots  z_{I,d})$ be its centre. Then,  on applying Lemma~\ref{usefullemmahigherdim}  (with $B=S=I$) and using \eqref{slv22}, we obtain  that  
\begin{equation}\label{slv2}
I \cap  T^{-n} \left( \prod_{i=1}^d B \left(z_{I,i}, (1 - \epsilon_n) \psi_i(n) \right) \right) \subset
I \cap A_n  \subset    I \cap  T^{-n} \left( \prod_{i=1}^d B \left(z_{I,i}, (1 + \epsilon_n) \psi_i(n) \right) \right) .  
\end{equation}
Now the left hand side of \eqref{slv2}
together with exponential mixing (Proposition~\ref{propositionexpmix}), implies that 
\begin{eqnarray*}
\mu (  I \cap A_n )  & \ge &      \mu  \Big(   I \cap T^{-n} \left( \textstyle{\prod_{i=1}^d}B \big(z_{I,i}, (1 - \epsilon_n) \psi_i(n) \big) \right)  \Big)  \\[1ex]
& \ge  &  \big(  \mu( I )  - C \gamma^n \big)   \  \mu  \Big(  \textstyle{\prod_{i=1}^d}B \big(z_{I,i}, (1 - \epsilon_n) \psi_i(n) \big)  \Big)  \\[1ex]
& \ge  &    \mu(I) 2^d (1-\epsilon_n)^d \psi(n) - C \gamma^n 2^d  \psi(n) 
\, .
\end{eqnarray*}
Thus, on using the fact that $(1-\epsilon_n)^d  \ge 1- d \epsilon_n$, it follows that 
\begin{eqnarray*}
\mu (A_n) = \sum_{ I \in \II} \mu (  I \cap A_n )  & \ge &   2^d (1-\epsilon_n)^d \psi(n) - \# \II \  C \gamma^n 2^d  \psi(n)
 \\[1ex]
& \ge  &   2^d\psi(n)  - 2^d d  \epsilon_n \psi(n)  \ - \   2^{2d}  C \gamma^n   \frac{1}{\epsilon_n^{d}}
\, .
\end{eqnarray*}
On using the right  hand side of \eqref{slv2}, and appropriately modifying the above  argument, we obtain  the  following  complementary upperbound estimate
\begin{eqnarray*}
\mu (A_n)   \ \le  \ 2^d\psi(n)  + 2^{2d}   \epsilon_n \psi(n)  \ + \   2^{3d}  C \gamma^n   \frac{1}{\epsilon_n^{d}}
\, .
\end{eqnarray*}
The upshot is that 
$$
\big|\mu (A_n)  -  \ 2^d\psi(n) \big|   \  \le  \    2^{2d}   \epsilon_n \psi(n)  \ + \   2^{3d}  C \gamma^n   \frac{1}{\epsilon_n^{d}}
\,
$$
and the desired statement follows on putting $ \epsilon_n = n^{-2}$ and observing that  both $ \sum_{n \in \N} \epsilon_n $  and  $ \sum_{n \in \N} \gamma^n/ \epsilon_n^d $ are convergent sums.

\subsection{Calculating the measure of $A_m \cap A_n$ on average} \label{interhd}
The goal of this section is to prove Proposition~\ref{measinterSV}.

\medskip

Let  $m < n  $.  To  start with we fix  a rectangle of linearity $J_m = \prod_i J_m^{(i)}$ for $T^m$ and
 estimate  the measure of   $ I_{J_m} \cap A_n :=  (J_m \cap A_m ) \cap A_n $. As in the proof of Proposition~\ref{measAnSV},    in what follows  $\{\epsilon_n\}_{n \in \N}$ is  a sequence of real numbers in $(0,1)$  that we specify later.

Let $I_{J_m} = J_m\cap A_m$. By Lemma \ref{lemmaIJm}, we have that  $I_{J_m} = \prod_{i=1}^d I_{J_m}^{(i)}$, where each $I_{J_m}^{(i)}$ is an interval of length $\ell_n^{(i)}   \leq 2\psi_i(m)/|K_{J_m}^{(i)} - 1|$. For each $1 \le i \le d$,  the interval $I_{J_m}^{(i)}$ can be written as the union of  
\[
\left\lceil
\frac{2|K_{J_m}^{(i)} - 1|^{-1} \psi_i(m)}{\epsilon_n \psi_i(n)}
\right\rceil  =: \NN_{J_m}^{(i)}
\]
disjoint intervals of equal length ${l}_n^{(i)} $. Indeed,  
\begin{equation} \label{ned} {l}_n^{(i)}   \ :=  \   \frac{\ell_n^{(i)}}{\NN_{J_m}^{(i)} } 
\ \le \ \epsilon_n \psi_i(n)  \,   . \end{equation}   Thus, $I_{J_m}$ can be written as the union of  
\[
\prod_{i=1}^d \left\lceil
\frac{2|K_{J_m}^{(i)} - 1|^{-1} \psi_i(m)}{\epsilon_n \psi_i(n)}
\right\rceil
\]
disjoint rectangles $I$ whose $i$-dimensional sidelength   is $ {l}_n^{(i)} $ and satisfies \eqref{ned}.     Fix such a rectangle $I = \prod_i I^{(i)}$  and let $z_I$ be its centre.
 In view of \eqref{ned} we are guaranteed that $\psi_i(n)  > {l}_n^{(i)} $   ($1 \le i \le d$),  and  thus  on applying Lemma~\ref{usefullemmahigherdim}  (with $B=S=I$)  we obtain that
\begin{eqnarray*}
I \cap A_n  & \subset  &   I \cap  T^{-n} \left( \prod_{i=1}^d B \left(z_{I,i}, (1 + \epsilon_n) \psi_i(n) \right) \right) .
\end{eqnarray*}
It follows that
\[
T^m(I\cap A_n) \subset T^m(I) \cap T^{-(n-m)}\left( \prod_{i=1}^d B \left(z_{I,i}, (1 + \epsilon_n) \psi_i(n) \right) \right)  \, . 
\]
This together with   exponential mixing (Proposition~\ref{propositionexpmix}) implies that 
\[
\mu \big(  T^m(I \cap  A_n)   \big) \leq \Big(\mu\big(T^m(I)\big) + C \gamma^{n-m}\Big)  \ \prod_{i=1}^d \big(2(1+\epsilon_n)\psi_i(n)\big)  \,  , 
\]
and  on summing over all rectangles $I$ in the disjoint partition of  $I_{J_m}$),   we find that 
\begin{eqnarray*}
\mu \big( T^m(I_{J_m}  \cap A_n) \big)  \ \leq  \  \left( \mu(T^m(I_{J_m})) + C \gamma^{n-m}\prod_{i=1}^d \left\lceil
\frac{2|K_{J_m}^{(i)} - 1|^{-1} \psi_i(m)}{\epsilon_n \psi_i(n)}
\right\rceil\right) \big(2(1+\epsilon_n)\big)^d \psi(n)    \, .  
\end{eqnarray*}
\\
Now  for any measurable set  $S \subseteq J_m $, we have that  $
 \mu(T^m(S) )     \, = \, |K_{J_m}|  \;    \mu(S) $, and so   it follows that 
\begin{eqnarray*}
\mu \big( I_{J_m}  \cap A_n \big)  \ \leq  \   \left( \mu(I_{J_m}) + C |K_{J_m}|^{-1}\gamma^{n-m}\prod_{i=1}^d \left\lceil
\frac{2|K_{J_m}^{(i)} - 1|^{-1} \psi_i(m)}{\epsilon_n \psi_i(n)}
\right\rceil\right) \big(2(1+\epsilon_n)\big)^d \psi(n).
\end{eqnarray*}

We are now in the position to calculate $\mu \big( A_m\cap A_n \big)$ by  summing  over all cylinder sets $J_m  \in \PP_m$. Before doing this,  we recall that
$$
|K_{J_m} |  = 1/ \mu( {J_m} )  \quad  { \rm  \  and  \  so\ } \quad   \sum_{J_m \in \PP_m }  |K_{J_m}|^{-1}   \ =  \ 1 \, , 
$$
and by  \eqref{gammadefsv}, for each $1 \le i\le d$, we have that 
$
|K_{J_m}^{(i)}-1|^{-1} \leq C_2 \gamma^m  $  where $ C_2:= (1 - \gamma)^{-1}  
$.
With this in mind,  together with the fact that $\psi_i(n) < 1 $  for $ n \in \N$ and $1 \le i \le d$,  it follows that 
\begin{align}  \label{tgb} 
\mu(A_m\cap A_n) &\leq \left( \mu(A_m) + C \gamma^{n-m}\prod_{i=1}^d \left\lceil
\frac{2C_2\gamma^m \psi_i(m)}{\epsilon_n \psi_i(n)}
\right\rceil\right) \big(2(1+\epsilon_n)\big)^d \psi(n) \nonumber   \\[2ex]
&\leq \left( \mu(A_m) + C \gamma^{n-m}\prod_{i=1}^d \left(1 + \frac{2C_2 \gamma^m \psi_i(m)}{\epsilon_n \psi_i(n)}
\right)\right) \big(2(1+\epsilon_n)\big)^d \psi(n) \nonumber  \\[2ex]
&\leq \left( \mu(A_m) + 2^d C \gamma^{n-m}\left(1 + (2C_2)^d \gamma^m \prod_{i=1}^d \frac{1}{\epsilon_n \psi_i(n)}
\right)\right) \big(2(1+\epsilon_n)\big)^d \psi(n) \nonumber \\[2ex]
&\leq  \mu(A_m) 2^d(1+2^d\epsilon_n) \psi(n)  +  2^d C \gamma^{n-m}\left(1 + (2C_2)^d \gamma^m \prod_{i=1}^d \frac{1}{\epsilon_n \psi_i(n)}
\right) 2^{2d} \psi(n) \nonumber  \\[2ex]
&\leq \mu(A_m)2^d \psi(n) + C_3 \epsilon_n \mu(A_m) \psi(n) + C_4 \gamma^{n-m} \psi(n) + C_5
\frac{\gamma^n}{\epsilon_n^d}  \, , 
\end{align}
where $C_3, C_4, C_5 > 0$ are absolute  constants. Now put  $\epsilon_n=n^{-2}$ for  $n \in \N$.  Then for arbitrary $a, b \in \N $ with $a < b$, it follows from \eqref{tgb} that 
\begin{align}  \label{doit}
2\mathop{\sum\sum}_{a\leq m<n\leq b} \mu(A_m\cap A_n)
& \ \leq  \  2\mathop{\sum\sum}_{a\leq m<n\leq b}\mu(A_m)2^d\psi(n) \  +  \ 2C_3 \sum_{a\leq m < b} \mu(A_m) \  \sum_{n=1}^\infty   n^{-2}  \nonumber  \\[1ex]
& \hspace*{20ex}  +  \ 2 C_4 \sum_{a < n \leq b} \psi(n) \sum_{m=1}^\infty \gamma^m  \ +  \ 2 C_5 \sum_{n=1}^\infty n^{2d+1}  \gamma^n   \nonumber \\[2ex]
& \ \leq  \  2\mathop{\sum\sum}_{a\leq m<n\leq b}\mu(A_m)2^d \psi(n) \  +  \ C'_3 \sum_{a\leq m < b} \mu(A_m)  \nonumber \\[0ex]
& \hspace*{20ex}  +  \  C'_4 \sum_{a < n \leq b} \psi(n)  \ +  \  C'_5  \, 
\end{align}
where $C'_3, C'_4, C'_5 > 0$ are absolute  constants.
It is now  reasonably straightforward  to verify that \eqref{doit}   together with Proposition~\ref{measAnSV}   implies the desired statement.

\vspace*{6ex}

\noindent{\it Acknowledgments}.  SV  would like to thank the one and only Bridget Bennett for being simply  ``lovely jubbly'' and oh yes ...  happy number sixty!  You are not only the GOAT in my life  but also ``In Our Time''.    Also many congratulations  to Lalji and Manchaben for making it into their nineties --  you  remain great role models and   what you guys have achieved in three continents is truly amazing!     Finally, thanks to Ayesha and Iona for being the best thing since slice bread and for allowing me to win from time to time when playing sports!    

The third-named author
was supported by a Royal Society University Research Fellowship, URF\tbs R1\tbs 180649.

\bibliographystyle{abbrv}

\bibliography{bibliography}
\end{document}